\documentclass[twoside,leqno]{article} 
\usepackage{amsfonts}
\usepackage{amsmath}
\usepackage{amssymb}
\usepackage{amsthm}
\usepackage{graphicx}
\usepackage{epsfig}
\usepackage{amsmath,amssymb,amsfonts}
\theoremstyle{definition}

\theoremstyle{theorem}
\newtheorem{theorem}{Theorem}[section]

\def\pam{\par\medskip}

\def\header#1#2#3#4#5{
    \markboth{#3}{#4}
    \setcounter{page}{#1}
    \title{#5}
    \author{#3}
    \date{}
    \mymaketitle{#1-#2}}
\newenvironment{abstractv}{\begin{quote}{\bf Abstract.\ }}{\end{quote}}
\def\msc{{\bf M.S.C. 2010}:\ }
\def\kwd{\\{\bf Key words}:\ }
\def\aua{\par\noindent{\em Author's address:}\pam\noindent}

\def\bece{\begin{center}}
\def\eece{\end{center}}
\def\bebi{\begin{thebibliography}{40}\setlength{\parskip}{-1mm}}
\def\eebi{\end{thebibliography}}

\makeatletter \@addtoreset{equation}{section} \makeatother
\def\mymaketitle#1{
\begin{document}\maketitle\thispagestyle{empty}}
%=========== HERE one can activate the line counter 2 =============
\def\babs{%\linenumbers
    \begin{abstractv}}
\def\eabs{\end{abstractv}}
%==================================================================

\header{1}{36}{M. Celli}
    {}
    {Vectors and a half-disk of triangle shapes in Ionescu-Weitzenb\"{o}ck's inequality}
		
\babs
The aim of this note is to give two new conceptual proofs of Ionescu-Weitzenb\"{o}ck's inequality. The first one, which is a vector proof, provides us a geometric interpretation of the difference between the two sides of this inequality and of two known corollary inequalities in differential geometry. Through a less classical approach, our second proof makes use of the properties of the set of triangles with a fixed ``size'', seen as a half-disk.
\eabs
\msc
    51M16.
\kwd
    Ionescu-Weitzenb\"{o}ck's inequality; sets of triangle shapes.
		
\section{Introduction}

Ionescu-Weitzenb\"{o}ck's classical inequality compares the area of a triangle and the sum of the squares of its sides:

\begin{theorem}
\label{theo1}
(Ionescu-Weitzenb\"{o}ck's inequality.) If \(a\), \(b\), \(c\) are the sides of a trian\-gle and
\(\Delta \) is its area, we have:
\[a^2+b^2+c^2\ge 4\sqrt{3}\Delta ,\]
with equality if \(a=b=c\).
\end{theorem}

An equivalent vector inequality is proved in \cite{stoica}:

\begin{theorem}
\label{theo2}
(Ionescu-Weitzenb\"{o}ck's vector inequality.) Let \(\vec{u}\), \(\vec{v}\) be two elements of a Euclidean vector space \(X\). Let us denote by \(\langle \vec{u},\vec{v}\rangle \) their inner product and by
\(\vec{u}\wedge \vec{v}\) their determinant in the plane they generate, oriented from \(\vec{u}\) to
\(\vec{v}\) (\(\vec{u}\wedge \vec{v}=0\) if \(\vec{u}\) and \(\vec{v}\) are collinear). We have:
\[||\vec{u}||^2+||\vec{v}||^2+||\vec{u}+\vec{v}||^2
\ge 2\sqrt{3}\vec{u} \wedge \vec{v}
=2\sqrt{3}\sqrt{||\vec{u}||^2||\vec{v}||^2-\langle \vec{u},\vec{v}\rangle ^2},\]
with equality if \(||\vec{u}||=||\vec{v}||=||\vec{u}+\vec{v}||\).
\end{theorem}

These theorems are equivalent. However, reference \cite{stoica} only mentions the implication
\ref{theo2}\(\Rightarrow \)\ref{theo1}.
In order to obtain Theorem \ref{theo1} from Theorem \ref{theo2}, denoting by \(A\), \(B\), \(C\)
the vertices of the triangle, we just have to apply it to \(\vec{u}=\overrightarrow{AB}\) and
\(\vec{v}=\overrightarrow{BC}\).
In order to obtain Theorem \ref{theo2} from Theorem \ref{theo1}, we need to consider the vector plane generated by \(\vec{u}\) and \(\vec{v}\) and the triangle \(ABC\),
where \(B=\vec{0}\), \(A=-\vec{u}\), \(C=\vec{v}\).\\

The main steps in the history of this inequality are described in \cite{stoica}. The most important proofs of this result and some of its generalizations can be found in this reference and its bibliography. In each of the next two sections, we will give a new conceptual proof. The first one is a vector proof, the second one is based on the properties of the set of triangles with a fixed ``size'', seen as a half-disk.

\section{The vector proof}

Using simpler arguments than in \cite{stoica}, we will prove the following generalization of Ionescu-Weitzenb\"{o}ck's vector inequality, which provides us a geometric interpretation of the difference between the two sides of the inequality
(it will turn out to be the term \(2||\vec{u}+R(\vec{v})||^2\)):

\begin{theorem}
\label{theo3}
Let \(\vec{u}\), \(\vec{v}\) be two elements of a Euclidean vector space \(X\). Let us denote by
\(\vec{u}\wedge \vec{v}\) their determinant in the plane they generate, oriented from
\(\vec{u}\) to \(\vec{v}\)
(\(\vec{u}\wedge \vec{v}=0\) if \(\vec{u}\) and \(\vec{v}\) are collinear).
Let \(R\) be the rotation of angle \(\pi /3\) of this oriented plane
(of every plane which contains \(\vec{v}\), with any orientation,
if \(\vec{u}\) and \(\vec{v}\) are collinear). We have:
\[||\vec{u}||^2+||\vec{v}||^2+||\vec{u}+\vec{v}||^2
=2(\sqrt{3}\vec{u} \wedge \vec{v}+||\vec{u}+R(\vec{v})||^2)
\ge 2\sqrt{3}\vec{u} \wedge \vec{v},\]
with equality if \(||\vec{u}||=||\vec{v}||=||\vec{u}+\vec{v}||\).
\end{theorem}		
		
\begin{proof}
Let us first notice that:
\[\vec{u} \wedge \vec{v}=||\vec{u}||||\vec{v}||\sin (\vec{u},\vec{v})
=-||\vec{u}||||\vec{v}||\cos \left ( (\vec{u},\vec{v})+\frac{\pi}{2}\right )\]
\[=-||\vec{u}||||R'(\vec{v})||\cos (\vec{u},R'(\vec{v}))
=-\langle \vec{u},R'(\vec{v}) \rangle,\]
where \(R'\) denotes the rotation of angle \(\pi /2\). Thus:
\[||\vec{u}||^2+||\vec{v}||^2+||\vec{u}+\vec{v}||^2-2\sqrt{3}\vec{u}\wedge \vec{v}
=2(||\vec{u}||^2+||\vec{v}||^2+\langle \vec{u},\vec{v} \rangle-\sqrt{3}\vec{u} \wedge \vec{v})\]
\[=2\left ( ||\vec{u}||^2+||\vec{v}||^2
+\left \langle 2\vec{u},\frac{1}{2}\vec{v}+\frac{\sqrt{3}}{2}R'(\vec{v})
\right \rangle
\right )\]
\[=2(||\vec{u}||^2+||R(\vec{v})||^2+2\langle \vec{u},R(\vec{v})\rangle )
=2||\vec{u}+R(\vec{v})||^2\cdot \]
The case of equality corresponds to: \(\vec{u}=-R(\vec{v})\).
This is equivalent to having the triangle
\((\vec{u}, \vec{v}, -(\vec{u}+\vec{v}))\)
be equilateral (\(||\vec{u}||=||\vec{v}||=||\vec{u}+\vec{v}||\)).
\end{proof}
		
Taking \(\vec{u}=\dot{r}(t)\), \(\vec{v}=-\ddot{r}(t)\) in this theorem, we obtain the following generalization of a theorem of \cite{stoica}, which corresponds to the inequality:
\[2\sqrt{3}K(t)\le 1+||\ddot{r}(t)||^2+||\dot{r}(t)-\ddot{r}(t)||^2\cdot \]

\begin{theorem}
\label{theo4}
Let \(r(t)\) be a parametrized curve in \(\mathbb{R}^3\), with constant velocity
\(||\dot{r}(t)||=1\) and curvature:
\[K(t)=\frac{||\dot{r}(t)\wedge \ddot{r}(t)||}{||\dot{r}(t)||^3}
=||\dot{r}(t)\wedge \ddot{r}(t)||,\]
where \(\wedge \) here denotes the vector cross product. We have:
\[2\sqrt{3}K(t)=1+||\ddot{r}(t)||^2+||\dot{r}(t)-\ddot{r}(t)||^2-2||\dot{r}(t)-R(\ddot{r}(t))||^2
\le 1+||\ddot{r}(t)||^2+||\dot{r}(t)-\ddot{r}(t)||^2,\]
with equality if \(||\ddot{r}(t)||=||\dot{r}(t)-\ddot{r}(t)||=1\),
where \(R\) denotes the rotation of angle \(\pi /3\) of the vector plane
generated by \(\dot{r}(t)\) and \(\ddot{r}(t)\),
oriented from \(\ddot{r}(t)\) to \(\dot{r}(t)\) (of every plane which contains \(\ddot{r}(t)\), with any orientation, if \(\dot{r}(t)\) and \(\ddot{r}(t)\) are collinear).
\end{theorem}

Applying Theorem \ref{theo3} to other vectors, as in \cite{stoica}, we can also obtain an identity involving the curvature and the torsion of a curve, which generalizes another inequality
of \cite{stoica}.

\section{The proof based on the half-disk of triangle shapes}

Our second proof of Theorem \ref{theo1} is based on a study of the properties of figures drawn on an abstract plane with coordinates \((a^2+b^2+c^2, \Delta )\), instead of the plane of the triangle of the theorem. In this plane, we will see the set of triangles with fixed \(a^2+b^2\) as a half-disk, and Ionescu-Weitzenb\"{o}ck's inequality as the inequation of a half-plane whose boundary is the tangent line to this half-disk which passes through the origin.

\begin{figure}[h]
\centering
\includegraphics[width=13cm]{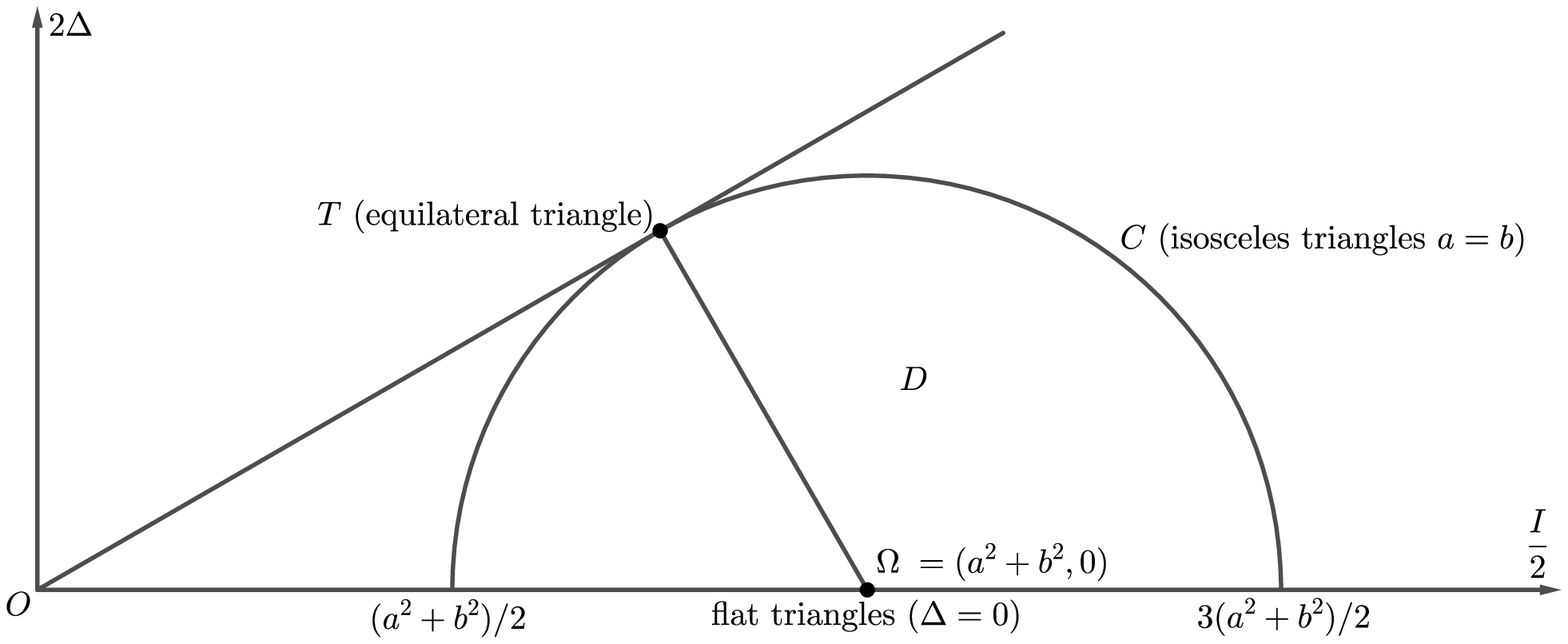}
\end{figure}

\begin{proof}
We have:
\[a^2b^2=(ab\cos (\gamma ))^2+(ab\sin (\gamma ))^2
=\left (\frac{a^2+b^2-c^2}{2} \right )^2+(2\Delta )^2,\]
where \(\gamma \) is the angle opposite the side of length \(c\),
\[\left ( \frac{I}{2}-(a^2+b^2)\right )^2+(2\Delta )^2=a^2b^2
\mbox{, where }I=a^2+b^2+c^2\cdot \]
Fixing the value of \(a^2+b^2\) and \(a^2b^2\), we can associate to each triangle a point 
with coordinates \((I/2>0,2\Delta >0)\) of the figure, located on the half-circle with equation:
\[F\left (\frac{I}{2},2\Delta \right )
=\left (\frac{I}{2}-(a^2+b^2)\right )^2+(2\Delta )^2
=a^2b^2\cdot \]
By the inequality \(ab\le (a^2+b^2)/2\), fixing only the value of \(a^2+b^2\), we can associate to each triangle a point of the half-disk \(D\) with inequation:
\[F\left (\frac{I}{2},2\Delta \right )
=\left (\frac{I}{2}-(a^2+b^2)\right )^2+(2\Delta )^2
\le \left (\frac{a^2+b^2}{2}\right )^2\]
in the quadrant \((I/2>0,2\Delta >0)\). This half-disk has center \(\Omega =(a^2+b^2,0)\)
and radius \((a^2+b^2)/2\). Let \(T\) be its point of contact with the tangent line which passes through the origin \(O\). We have:
\[\sin (\Omega OT)=\frac{\Omega T}{\Omega O}
=\frac{\frac{a^2+b^2}{2}}{a^2+b^2}=\frac{1}{2}\cdot \]
Thus: \(\Omega OT=\pi /6\), and the slope of the tangent line \(OT\) is:
\(\tan (\Omega OT)=1/\sqrt{3}\).
As the half-disk \(D\) is below \(OT\), we have:
\[\frac{2\Delta }{\frac{I}{2}}\le \frac{1}{\sqrt{3}},\]
which gives us Ionescu-Weitzenb\"{o}ck's inequality.\\
The case of equality corresponds to the point \(T\), located on the limit half-circle \(C\) with equation:
\[F\left (\frac{I}{2},2\Delta \right )=a^2b^2
=\left (\frac{a^2+b^2}{2}\right )^2\cdot \]
The equality between the two last expressions is equivalent to: \(a=b\). In other words, the limit half-circle \(C\) is the set of isosceles triangles with base the side of length \(c\). By symmetry, we have, for \(T\): \(a=b=c\). Thus, the case of equality corresponds to the equilateral triangle.
\end{proof}

The equation of the circles of this proof arises from the following system:
\[\left \{
\begin{array}{c}
ab\cos (\gamma )=\frac{a^2+b^2-c^2}{2}\\
ab\sin (\gamma )=2\Delta
\end{array}
\right .\]
In fact, the use of the expressions \(\langle \vec{u},\vec{v} \rangle \)
(instead of \(ab\cos (\gamma )\))
and \(\vec{u}\wedge \vec{v}\) (instead of \(ab\sin (\gamma )\)) in our proof of Theorem \ref{theo3} shows that it was implicitly based on equivalent equations.\\

There are other descriptions of the set of triangles. This note is also an invitation to read
\cite{chenciner}, where the set of triangles with fixed \(I=a^2+b^2+c^2\) is seen as a sphere,
in order to solve problems of mechanics.

\bebi

\bibitem{chenciner}
    A. Chenciner,
    {\em The ``form'' of a triangle},
		Rend. Mat. Appl. (7) 27 (2007), 1-16.\\
		http://www1.mat.uniroma1.it/ricerca/rendiconti/ARCHIVIO/2007(1)/1-16.pdf

\bibitem{stoica}
    E. Stoica, N. Minculete, C. Barbu,
	  {\em New aspects of Ionescu-Weitzenb\"{o}ck's ine\-quality},
	  Balk. J. Geom. Appl. 21, 2 (2016), 95-101.\\
		http://emis.ams.org/journals/BJGA/v21n2/B21-2st-b21.pdf

\eebi

\aua
    Martin Celli\\
    Departamento de Matem\'aticas,\\
	  Universidad Aut\'onoma Metropolitana-Iztapalapa,\\	
		Av. San Rafael Atlixco 186, Col. Vicentina, Alc. Iztapalapa, Mexico City, CP 09340, Mexico.\\
    E-mail: cell@xanum.uam.mx

\end{document}